\documentclass[reqno]{amsart}
\usepackage{amsfonts}%
\usepackage{amssymb}
\usepackage{mathrsfs}
\usepackage{latexsym,amsmath,amsfonts}
\usepackage[backref]{hyperref}
\usepackage{cite}
\usepackage{color}
\allowdisplaybreaks
\date{\today}

\theoremstyle{plain}
\newtheorem{theorem}{Theorem}[section]
\newtheorem{cor}[theorem]{Corollary}
\newtheorem{lem}[theorem]{Lemma}
\newtheorem{prop}[theorem]{Proposition}
\theoremstyle{definition}

\theoremstyle{remark}
\newtheorem{rem}{Remark}[section]
\numberwithin{equation}{section}

\def\be{\begin{eqnarray}}
\def\ee{\end{eqnarray}}

\def\pt{\partial_t}
\def\r{{\mathbb{T}}^2}

\def\md{\mbox{div}}

\begin{document}

\title[Asymptotic   limits of Quantum Navier-Stokes equations ]
{Inviscid, incompressible and semiclassical limits of Quantum Navier-Stokes equations }

\author{Hongli Wang}
\address{School of Mathematics and Statistics,
North China University of  Water Resources and Electric Power,
                                      Zhengzhou 450045, Henan Province, P. R.
                                      China   }
 \email{wanghongli@ncwu.edu.cn}


\author{Jianwei Yang }
\address{School of Mathematics and Statistics,
North China University of  Water Resources and Electric Power,
                                      Zhengzhou 450045, Henan Province, P. R.
                                      China   }
 \email{ yangjianwei@ncwu.edu.cn}



\begin{abstract}

In the paper, we consider the inviscid, incompressible and semiclassical limits  limits of the barotropic quantum Navier-Stokes equations of compressible
flows in a periodic domain.  We show that the limit solutions satisfy the incompressible Euler system  based on the relative entropy inequality and on the detailed analysis for
general  initial data. The rate of convergence is estimated in terms of the Mach number.

\end{abstract}

\maketitle

\section{ Introduction}

\quad This paper is concerned with the following non-dimensional
 compressible Quantum-Navier-Stokes (QNS) 
system on $(0, T)\times\Omega$,
\begin{align}
             &\pt n+\md (nu)=0, \label{qmhd1}\\
             &\displaystyle\pt (nu)+\md(nu\otimes u)+\nabla p(n)
 -2\kappa^2n\nabla
 \left(\frac{\Delta\sqrt{n}}{\sqrt{n}}\right)
 =2\nu\md(nD(u)).\label{qmhd2} 
\end{align}
The domain $\Omega$ we consider is the $2$-dimensional torus $\r$.
There are  two unknowns, the density $n=n(t,x)$ and the velocity field $u=u(t,x)$   of the fluid. $\nu$ and $\kappa$ are positive constants and they are called the viscosity and the dispersive coefficients.
$p(n)$ is the pressure, and in this paper, we consider
the case of isentropic flows with $p(n)=n^{\gamma}$ for
$\gamma>1$.
 In \eqref{qmhd2}, $D(u)=(\nabla u+\nabla u^T)/2$. The term $2\kappa^2n\nabla
 \left(\Delta\sqrt{n}/\sqrt{n}\right)$ can be interpreted as the
quantum Bohm potential term, or as a quantum correction to the
pressure. Moreover,  the following relation
holds\begin{align}\label{ret1}
    2n\nabla
 \left(\frac{\Delta\sqrt{n}}{\sqrt{n}}\right)
 =\nabla\Delta n-4\md(\nabla \sqrt{n}
 \otimes\nabla \sqrt{n}),
\end{align}which can avoid using too high regularities of the
density
 $n$.  Brull and M\'{e}hats \cite{bm09} utilized a moment
method and a Chapman-Enskog expansion around the quantum equilibrium
to derive \eqref{qmhd1}-\eqref{qmhd2} from a Wigner equation.

The main purpose of this paper is to rigorously prove  the combined
incompressible and inviscid limits in the framework of the
global weak solutions to \eqref{qmhd1}-\eqref{qmhd2}.  To begin
with, we introduce the scaling
\begin{gather*}\label{naaz}
 t\mapsto \epsilon t, u\mapsto\epsilon u, \end{gather*}
and set
    $$\kappa=\epsilon^{2}, \nu=\epsilon^2,$$
 where $0<\epsilon<1$ is a small parameter proportional to the Mach number. With such scalings, the quantum
Navier-Stokes equations \eqref{qmhd1}-\eqref{qmhd2} read as
\begin{align}
 & \pt n +\md (nu)=0,\label{n} \\
  & \displaystyle\pt (nu)
  +\md(nu\otimes u)+\frac{1}{\epsilon^2}\nabla p(n)
  \notag\\
  &\quad\quad\quad
 -2\epsilon^{2}\nabla
 \left(\frac{\Delta\sqrt{n}}{\sqrt{n}}\right)
 =2\epsilon\md(nD(u)).\label{u}
\end{align}

Our aim is to identify the scaled QNS system \eqref{n}-\eqref{u}
in the limit $\epsilon\to0$, meaning the incompressible, inviscid and  semiclassical  limit. More precisely,  we want to prove that the weak solution of the QNS system  \eqref{n}-\eqref{u} converges to the classical solution of the corresponding incompressible Euler system, namely
\begin{align}
    \pt v+(v\cdot\nabla)v+\nabla \Pi=
   0,\quad\md v=0.\label{2018061001}
\end{align}

Multiplying equation \eqref{u} by $u$ and integrating by parts, we obtain the energy inequality of  the QNS system\eqref{n}-\eqref{u} in its integral form
\begin{equation}\label{qmhdener}
 {E}(t)+2\epsilon \int^{t}_{0}\int_{\r}n|D(u)|^2dxdt\leq E(0),\end{equation}
 where the total energy $E$ is given  by the sum of the
kinetic, internal and quantum energy:
\begin{equation*}
E(t)=\int\left(\frac12n|u|^2+\frac{1}{\epsilon^{2}}H(n)+2\epsilon^{2}|\nabla\sqrt{n}|^2\right)dx,
\end{equation*} 
with  
$$E(0)=\int\left(\frac12n_{0}|u_{0}|^2+\frac{1}{\epsilon^{2}}H(n_{0})+2\epsilon^{2}|\nabla\sqrt{n_{0}}|^2\right)dx$$
the initial energy, and 
\begin{equation}\label{2017070711}H(n)=\frac{1}{\gamma-1}(n^{\gamma}-\gamma(n-1)-1)=\frac{1}{\gamma-1}(p(n)-p'(1)(n-1)-p(1))\end{equation}
the Helmhoitz free energy.

Without  quantum effects, system \eqref{n}-\eqref{u} reduces to the Navier-Stokes equations,  whose incompressible inviscid limit was investigated by Lions-Masmoudi \cite{lm98} in the case of well-prepared initial data and by  Masmoudi \cite{m01}  in the case of ill-prepared initial data in the whole space case and also in the periodic case.  
In \cite{cn2017}, Caggio-Nečasová  consider the inviscid incompressible limits of the rotating compressible Navier-Stokes system for a barotropic fluid and show that the limit system is represented by the rotating incompressible Euler equation on the whole space.
The present paper will extend the results in \cite{m01}
 to the 
quantum Navier-Stokes equations.  In comparison with \cite{lm98,m01,cn2017},  the present problem features some additional mathematical difficulties related  to  the third-order
 derivative term in the momentum equations.  
 Next, the viscosity we consider
 here is dependent on the density.    Finally, we use the refined energy analysis to obtain the desired rate of convergence.      Therefore new
 techniques and ideas are introduced to treat them.
Our approach is based on  the existence of global in time \emph{finite energy weak solutions}   \cite{as2017}  for the Cauchy problem  of QNS system \eqref{n}-\eqref{u} and the recently discovered relative entropy inequality \cite{fn2012,g2011} which gives us a very powerful and concise tool for the purpose of measuring the stability of a solution compared to another solution with a better regularity. We will prove  that   the convergence of QNS system to the limit problem (incompressible Euler system) takes place on any time interval $[0, T]$ on which the Euler system \eqref{2018061001} possesses a regular solution by introducing a new  relative entropy functional.  

Recently, we saw a publication \cite{kl2018} by Kwon-Li  in which a similar result was proved on the convergence of the degenerate quantum compressible Navier-Stokes equations with damping to the incompressible Navier-Stokes/Euler equations. Here we state the main differences between two papers. First,  the relative entropy functional we introduced  is  very different  from that of  \cite{kl2018}.  Second, the initial condition  which is needed in our main results is different from that of their main results in  \cite{kl2018}. Finally, The rate of convergence in two papers is also different. We also mention that the incompressible limits of  the compressible Navier-Stokes system and related models are very interesting and there are a lot of references  on this topic.   See Alazard \cite{a06} for Navier-Stokes equations, Feireisl-Novotný \cite{fn2013,fn2014} for the Navier-Stokes-Fourier system and the Euler-Boussinesq System, Jiang-Ou \cite{jo11} for the non-isentropic Navier-Stokes equations, Hu-Wang \cite{hw2009} 
for the viscous compressible magnetohydrodynamic flows, Jiang-Ju-Li \cite{jjl2010} for the compressible magnetohydrodynamic equations with vanishing viscosity coefficients, Ukai \cite{u86} for the Euler equations, Wang-Yu \cite{wy2014} for the compressible flow of Liquid crystals and references therein.

In this present paper, we denote by  $\textbf{1}$  the characteristics
function and  $C$ the generic positive constants independent of
$\epsilon$.

The rest of this paper is organized as follows. In the next section,
we state some useful known results and 
 our main results. Finally, Section 3
  is devoted to the proof of our main result.

\section{ Main Results}

\indent \quad In this section, we state our main results. For this, we first recall the following classical result on the existence of sufficiently regular solutions of the incompressible Euler system \eqref{2018061001} with the initial data $v(0)=v_{0}$.
\begin{prop}
\label{prop2018061002} (Ref.\cite{K, Mc}) Let $v_{0}\in C^{\infty}(\r)$ satisfying $\emph{\md}{v}_{0}=0$. Then, for any $T>0$,  the incompressible Euler system \eqref{2018061001} exists  a unique classical  solution $v$ satisfying  \begin{equation}
v,\Pi\in C^{\infty}([0,T]\times\r).\label{ie}
\end{equation}  \end{prop}

Recently, in \cite{as2017} Antonelli-Spirito consider the QNS system\eqref{n}-\eqref{u} both in two and in three space dimensions and prove the following result on the global existence of finite energy weak solutions for large initial data.  \begin{prop} [Ref. \cite{as2017}]
 Let $\alpha$ be a small fixed positive number and the initial data $n_{0},u_{0}$ satisfying 
 \begin{align}
    &0< \frac{1}{\overline{n}_{0}}< n_{0}<\overline{n}_{0}, n_{0}\in L^{1}(\r)\cap L^{\gamma}(\r), \nabla\sqrt{n_{0}}\in L^{2}(\r)\cap L^{2+\alpha}(\r),\label{2018061006}\\
    & \sqrt{n_{0}}u_{0}\in L^{2}(\r)\cap L^{2+\alpha}(\r) \label{2018061007}\end{align}
 for a positive constant $\overline{n}_{0}$. 
 Then for any $0 < T <+\infty$ , there exists a finite energy weak solutions $(n^{\epsilon}, u^{\epsilon})$ of the compressible QNS system \eqref{n}-\eqref{u}   with initial data $n_{0},u_{0}$   on $(0, T]\times \r$  satisfying 
 \begin{itemize}
\item Integrability conditions:
\begin{equation*}
n\in L^{\infty}(0,T; L^{\gamma}(\r)),\; \sqrt{n}u\in L^{\infty}(0,T; L^{2}(\r)),\; \sqrt{n}\in L^{\infty}(0,T; H^{1}(\r)).\end{equation*}
\item Continuity equation:
\begin{equation*}
\int_{\r} n_{0}\phi(0)dx+\int^{T}_{0} \int_{\r} n\phi_t+\sqrt{n}\sqrt{n} u\nabla\phi dxdt=0,
\end{equation*}
for any $\phi\in C_0^{\infty}([0,T]\times (\r))$.
 \item  Momentum equation:
\begin{equation*}
\begin{aligned}
&\int_{\r}n_{0}u_{0}\psi(0)+\int^{T}_{0} \int_{\r}\sqrt{n}\sqrt{n} u\psi_tdxdt\\&+\int^{T}_{0} \int_{\r} \sqrt{n} u \otimes \sqrt{n} u\nabla\psi dxdt+\frac{1}{\epsilon^{2}}\int^{T}_{0} \int_{\r} p(n)\emph{\md}\psi dxdt\\
&-2\epsilon\int^{T}_{0} \int_{\r} (\sqrt{n} u\otimes\nabla\sqrt{n})\nabla\psi dxdt-2\epsilon\int^{T}_{0} \int_{\r} (\nabla\sqrt{n} \otimes\sqrt{n} u)\nabla\psi dxdt\\
&+\epsilon\int^{T}_{0} \int_{\r} \sqrt{n}\sqrt{n} u\Delta\psi dxdt+\epsilon\int^{T}_{0} \int_{\r} \sqrt{n}\sqrt{n} u\nabla\emph{\md}\psi dxdt\\
&-4\epsilon^{2}\int^{T}_{0} \int_{\r} t(\nabla\sqrt{n}\otimes\nabla\sqrt{n} )\nabla\psi dxdt+2\epsilon^{2}\int^{T}_{0} \int_{\r} \sqrt{n}\nabla\sqrt{n}\nabla\emph{\md}\phi dxdt=0,\\
\end{aligned}
\end{equation*}
for any $\psi\in C_0^{\infty}([0,T]\times (\r))$. 
\item Energy inequality: if
\begin{equation*}
E(t)=\int_{\r}\left(\frac12n|u|^2+\frac{1}{\epsilon^{2}}H(n)+2\epsilon^{2}|\nabla\sqrt{n}|^2\right)dx,
\end{equation*} 
then the following energy inequality is satisfied for a.e. $t\in[0,T]$
\begin{equation*}
E(t)\leq E(0).
\end{equation*}
\end{itemize}

   \end{prop}

\begin{rem}
 The scaling we introduced for QNS system   satisfies the requirement to  the viscosity and the dispersive coefficients.\end{rem} 
Motivated by \cite{fn2012},  we introduce the  following relative entropy functional
\begin{align} 
\label{2018061101}
    \mathcal{E}(t)=& \int_{\r}\left[\frac{1}{2}n|u-v-\nabla\Psi|^{2}+2\epsilon^{2}|\nabla\sqrt{n}-\nabla\sqrt{1+\epsilon\sigma}|^2\right.\notag\\
    &+\left.\frac{1}{\epsilon^{2}}[H(n)-H'(1+\epsilon\sigma)(n-1-\epsilon \sigma)-H(1+\epsilon \sigma)]\right]dx,  \end{align}
where $\sigma,\Psi$ are the solution of the following acoustic system related to the QNS system \eqref{n}-\eqref{u} by the following linear relations \cite{fns2014}\begin{align}
& \partial_{t}\sigma+\frac{1}{\epsilon}\Delta\Psi=0,\label{2018062501}\\
&  \partial_{t}\nabla\Psi+\frac{p'(1)}{\epsilon}\nabla\sigma=0\label{2018062502}\end{align}
supplemented with the initial data
\begin{equation}
\label{2018062503}
\sigma(0)=\sigma_{0}=n^{1}_{0},\quad \nabla\Psi(0)=\nabla\Psi_{0}=u_0-v_0,
\end{equation}
where $v_{0}=\textbf{H}[u_{0}]$ denotes
the Helmholtz projection into the space of solenoidal functions. 
Similarly to \cite{fn2013}, our goal is to apply a Gronwall-type argument to the relative entropy inequality \eqref{2018070102}to deduce the strong convergence to the limit system claimed in Theorem \ref{th2018070308}.
The initial data $(n_{0}^{1},\nabla\Psi_{0})$  can be regularized in the following way
\begin{equation}
\label{2018062510}
n_{0}^{1}=n_{0,\eta}^{1}=\chi_{\eta}\star (\psi_{\eta}n_{0}^{1}),\quad \nabla\Psi_{0}=\nabla\Psi_{0,\eta}=\chi_{\eta}\star (\psi_{\eta}\nabla \Psi_{0}),\quad \eta>0,\end{equation}
where $\{\chi_{\eta}\}$ is a family of regularizing kernels and $\psi_{\eta}\in C^{\infty}_{0}(\r)$ is standard cut-off function. 
The total change in energy of the fluid caused by acoustic wave is give by 
\begin{equation*}
\frac12\int_{\r}(p'(1)|\sigma|^{2}+|\nabla\Psi|^{2})dx,
\end{equation*}
which is conserved in time, namely 
\begin{equation}
\label{2018062511}
\frac12\int_{\r}(p'(1)|\sigma|^{2}+|\nabla\Psi|^{2})dx=\frac12\int_{\r}(p'(1)|\sigma_0|^{2}+|\nabla\Psi_0|^{2})dx.\end{equation}
 In addition,  for any $t>0$, the dispersive estimates hold \cite{fn2013}
\begin{equation}
\label{2018062512}
\|\nabla\Psi\|_{W^{k,p}}+\|\sigma\|_{W^{k,p}}\leq C\left(1+\frac{t}{\epsilon}\right)^{\frac{1}{p}-\frac{1}{q}}\left(\|\nabla\Psi_{0}\|_{W^{k,p}}+\|n^{1}_0\|_{W^{k,p}}\right),\end{equation}
where   $$2\leq p\leq+\infty, \quad \frac{1}{p}+\frac{1}{q}=1,\quad k=0,1,\cdots.$$

    The main result of this paper can be stated as
follows.
\begin{theorem} \label{th2018070308} Let $M>0$ be a constant. Let the initial data of the compressible QNS system \eqref{n}-\eqref{u} be of the following form
\begin{align}
  & n(0)= n_{0,\epsilon}  =1+\epsilon n_{0,\epsilon}^{1},\quad u(0)=u_{0,\epsilon},\label{2018070301}\\
  & \|n_{0,\epsilon}^{1}\|_{H^1(\r)}+\|u_{0,\epsilon}\|_{L^{2}(\r)}\leq M  \label{2018070302}\end{align}
and satisfying \eqref{2018062510} for fixed $\eta$. Let  all
 requirements of  Proposition \ref{prop2018061002} be satisfied with the initial data $v_{0}=\mathbf{H}[u_{0}]$ for the Euler system \ref{2018061001}. Let $(\sigma,\Psi)$ be the solution of the acoustic system  \eqref{2018062501}-\eqref{2018062502} with the initial data \eqref{2018062510}.  Then, for any $t\in[0,T]$ and any weak solutions $(n,u)$ of  compressible QNS system \eqref{n}-\eqref{u},  we have 
 \begin{align}\label{2018071401}
 &\|\sqrt{n}(u-v-\nabla\Psi)\|^{2}_{L^{2}(\r)}+\left\|\frac{n-1-\epsilon\sigma}{\epsilon}\right\|^{2}_{L^{2}(\r)}\notag\\&+\epsilon^{2}\|\nabla\sqrt{n}-\nabla\sqrt{1+\epsilon\sigma}\|^2_{L^{2}(\r)}\notag\\
 \leq&C\left(\|u_{0,\epsilon}-u_{0}\|^2_{L^{2}(\r)}+\|n^{1}_{0,\epsilon}-n^{1}_{0}\|^2_{H^{1}(\r)}+ \epsilon^{\min\{1-\frac{1}{\gamma},\frac{1}{\gamma}\}}\right) ,\end{align}
for a.e. $t\in[0,T]$.
\end{theorem}

For any vector field $\phi$, we use $P$ and $Q$ to denote the divergence-free part of  $\phi$, and the gradient part of $\phi$, respectively, i.e. $P = I- Q$ and $Q = \nabla􏰓-\nabla\triangle^{-1} \md$.

The following corollary is a consequence of the above theorem. 

\begin{cor}\label{2018070906}
 In Theorem \ref{th2018070308}, we assume that 
  \begin{equation}
\label{2018070310}
\|u_{0,\epsilon}-u_{0}\|^2_{L^{2}(\r)}+\|n^{1}_{0,\epsilon}-n^{1}_{0}\|^2_{H^{1}(\r)}\leq C \epsilon^{\min\{1-\frac{1}{\gamma},\frac{1}{\gamma}\}}. \end{equation}Then, for any weak solutions $(n,u)$ of  compressible QNS system \eqref{n}-\eqref{u},  we have 
\begin{equation}
\label{2018070910}
\|P(\sqrt{n}u)-v\|^{2}_{L^{2}(\r)}+\left\|\frac{n-1}{\epsilon}\right\|^{2}_{L^{2}(\r)}+\epsilon^{2}\|\nabla\sqrt{n}\|^{2}_{L^{2}(\r)}\leq C\epsilon^{\min\{1-\frac{1}{\gamma},\frac{1}{\gamma}\}}
\end{equation}
for a.e.$t\in[0,T]$.
\end{cor}

\section{Proof of Theorem }

We consider the class of finite energy weak solutions of the compressible QNS system \eqref{n}-\eqref{u} satisfying, besides the standard weak formulation of the equations,    the energy inequality
\begin{align}
  &\int_{\r}\left\{\frac12n
|u|^2+
   \frac{1}{\epsilon^2}
   H(n)
   +2\epsilon^{2}|\nabla\sqrt
   {n}|^2\right\}dx+2\epsilon\int_0^t
   \int_{\r} n|D(u)|^2dxd\tau\notag\\
 \leq&\int_{\r}\left\{\frac12n_{0,\epsilon}
|u_{0,\epsilon}|^2+
   \frac{1}{\epsilon^2}
   H(n_{0,\epsilon})
   +2\epsilon^{2}|\nabla\sqrt
   {n_{0,\epsilon}}|^2\right\}dx
   \notag\\
  \leq& C.\label{nein}
\end{align}
Therefore, we have the following properties:
\begin{eqnarray}
   &&\sqrt{n}u\,\,\mbox{is}
   \,\,\mbox{bounded}\,\, \mbox{in}
   \,\,L^{\infty}([0,T];L^2(\r)), \label{a1}\\
   &&\frac{1}{\epsilon^2}
   H(n)\,\,\mbox{is}\,\,
   \mbox{bounded}\,\, \mbox{in}
   \,\,L^{\infty}([0,T];L^1(\r)), \label{a3} \\
   &&\sqrt{\epsilon n}D(u)\,\,\mbox{is}
   \,\,\mbox{bounded}\,\, \mbox{in}
   \,\,L^{2}([0,T];L^21(\r)).\label{a2}
\end{eqnarray}

\begin{lem}\label{Lem:2018070121}
Let $(n,u)$ be the weak solution
    to quantum Navier-Stokes
equations \eqref{n}-\eqref{u}on $[0, T]$. Then there exists a
constant $C
> 0$ such that for all $\epsilon\in(0,1)$ and $\gamma>1$,
  \begin{align}
&\left\|(n-1)\mathbf{1}_{\{|n-1|<1\}}\right\|_{L^{\infty}([0,T];L^2(\r))}
  \leq C\epsilon,\label{2018070122}\\
  &\left\|(n-1)\mathbf{1}_{\{|n-1|\geq1\}}\right\|_{L^{\infty}([0,T];L^{\gamma}(\r))}
  \leq C\epsilon^{\frac{2}{\gamma}}.\label{2018070409}  \end{align}
Furthermore, we have 
  \begin{equation}\label{2018071701}
\|n^{\epsilon}-1\|_{L^{\infty}([0,T];L^\gamma(\r))}
  \leq C\epsilon^{\frac{\lambda}{\gamma}}\quad \emph{and} \quad
  \|n^{\epsilon}-1\|_{L^{\infty}([0,T];L^{\lambda}(\r))}
  \leq C\epsilon,
  \end{equation}where, $\lambda=\min\{2,\gamma\}$.\end{lem}

\begin{proof} In view of the Lemma 5.3
in \cite{l98}, there exist two
 positive constants $c_1\in(0,1)$ and $c_2\in(1,+\infty)$
independent of $n$ such that the following inequality
\begin{align}
   &c_1\int_{\r}\left(|n-1|^2\mathbf{1}_{\{|n-1|<1\}}
    +|n-1|^\gamma\mathbf{1}_{\{|n-1|\geq1\}}\right)dx\notag  \\
   &\quad\leq\int_{\r}H(n)dx\leq c_2\int_{\r}\left(|n-1|^2
   \mathbf{1}_{\{|n-1|<1\}}
    +|n-1|^\gamma\mathbf{1}_{\{|n-1|\geq1\}}\right)dx,
\end{align} which implies the Lemma \ref{Lem:2018070121}
 holds
due to  \eqref{a3}. \end{proof}

With the help of  the smoothness of the limit system 
\eqref{2018061001}, we take $v+\nabla\Psi$ as the test function in the 
weak formulation of equation \eqref{u} to yield the following equality for almost all $t$,
\begin{align}
\label{2018062513}
   &\mathcal{E}(t)+2\epsilon\int^{t}_{0}\int_{\r}n|D(u-v-\nabla\Psi)|^{2}dxd\tau\notag  \\
    =&\underbrace{ \int_{\r}\left[\frac{1}{2}n|u|^{2}+2\epsilon^{2}|\nabla\sqrt{n}|^2+\frac{1}{\epsilon^{2}}H(n)\right]dx+2\epsilon\int^{t}_{0}\int_{\r}n|D(u)|^{2}dxd\tau}_{\leq\int_{\r}\left[\frac{1}{2}n_{0,\epsilon}|u_{0,\epsilon}|^{2}+2\epsilon^{2}|\nabla\sqrt{n_{0,\epsilon}}|^2+\frac{1}{\epsilon^{2}}H(n_{0,\epsilon})\right]dx}\notag  \\
    &- \int_{\r}nu\cdot(v+\nabla\Psi)dx+\frac12\int_{\r}n|v+\nabla\Psi|^2dx\notag  \\
    &-\frac{1}{\epsilon^{2}}\int_{\r}\left[H'(1+\epsilon\sigma)(n-1-\epsilon \sigma)+H(1+\epsilon \sigma)\right]dx\notag  \\&-4\epsilon^{2}\int_{\r}\nabla\sqrt{n}\cdot\nabla\sqrt{1+\epsilon\sigma}dx+2\epsilon^{2}\int_{\r}|\nabla\sqrt{1+\epsilon\sigma}|^{2}dx\notag\\
    &-4\epsilon\int^{t}_{0}\int_{\r}nD(u):D(v+\nabla\Psi)dxd\tau\notag  \\
    &+2\epsilon\int^{t}_{0}\int_{\r}n|D(v+\nabla\Psi)|^{2}dxd\tau.    \end{align}

We use $v+\nabla\Psi$ as a test function in the weak
formulation of momentum equation \eqref{u} to yield the following equality for almost all $t$:
\begin{align}
\label{2018062601}
&- \int_{\r}nu\cdot(v+\nabla\Psi)dx    \notag\\=&  - \int_{\r}n_{0,\epsilon}u_{0,\epsilon}\cdot(v_0+\nabla\Psi_0)dx-\int^{t}_{0}\int_{\r}nu\cdot\pt(v+\nabla\Psi)dxd\tau\notag\\
&-\int^{t}_{0}\int_{\r}nu\cdot\nabla(v+\nabla\Psi)\cdot udxd\tau\notag\\
&-\frac{1}{\epsilon^2} \int^{t}_{0}\int_{\r}p(n)\Delta\Psi dxd\tau+  \epsilon^{2}\int^{t}_{0}\int_{\r}(n-1)\triangle\Psi dxd\tau\notag\\&+4\epsilon^{2}\int^{t}_{0}\int_{\r}(\nabla\sqrt{n}\otimes\nabla\sqrt{n}):\nabla(v+\nabla\Psi)dxd\tau\notag\\
    &+2\epsilon\int^{t}_{0}\int_{\r}nD(u):D(v+\nabla\Psi)dxd\tau,\end{align}
    where we have used the equality \eqref{ret1}.
By a direct computation, we have
\begin{align}
\label{2018070101}
    &\frac12\int_{\r}n(v+\nabla\Psi)^2dx    \notag\\
    =&  \frac12\int_{\r}n_{0,\epsilon}(v_0+\nabla\Psi_0)^2dx + \int^{t}_{0}\int_{\r}nu\cdot\nabla(v+\nabla\Psi)\cdot(v+\nabla\Psi)dxd\tau\notag  \\
    &  + \int^{t}_{0}\int_{\r}n(v+\nabla\Psi)\pt(v+\nabla\Psi)dxd\tau\end{align}
and 
\begin{align}
\label{2018062611}
 &-\frac{1}{\epsilon^{2}}\int_{\r}\left[H'(1+\epsilon\sigma)(n-1-\epsilon \sigma)+H(1+\epsilon \sigma)\right]dx \notag\\
 =  & -\frac{1}{\epsilon^{2}}\int_{\r}\left[\epsilon H'(1+\epsilon n_{0}^{1})(n_{0,\epsilon}^{1}-n_{0}^{1})+H(1+\epsilon n_{0}^{1})\right]dx \notag \\
    &- \frac{1}{\epsilon^{2}}\int^{t}_{0}\int_{\r}(n-1-\epsilon\sigma) \pt H'(1+\epsilon\sigma)dxd\tau\notag\\&-\frac{1}{\epsilon^{2}}\int^{t}_{0}\int_{\r}nu\cdot\nabla H'(1+\epsilon\sigma)dxd\tau.\end{align}

Then, putting \eqref{2018062601}-\eqref{2018062611}  into the   \eqref{2018062513}, we have
\begin{align}
\label{2018070102}
    &\mathcal{E}(t)+2\epsilon\int^{t}_{0}\int_{\r}n|D(u-v-\nabla\Psi)|^{2}dxd\tau\notag  \\
    \leq&\mathcal{E}(0)-\int^{t}_{0}\int_{\r}n[\pt(v+\nabla\Psi)+u\cdot\nabla(v+\nabla\Psi)]\cdot(u-v-\nabla\Psi)dxd\tau \notag \\
    &- \frac{1}{\epsilon^{2}}\int^{t}_{0}\int_{\r}(n-1-\epsilon\sigma) \pt H'(1+\epsilon\sigma)dxd\tau\notag\\&-\frac{1}{\epsilon^{2}}\int^{t}_{0}\int_{\r}nu\cdot\nabla H'(1+\epsilon\sigma)dxd\tau\notag  \\
&-\frac{1}{\epsilon^2} \int^{t}_{0}\int_{\r}p(n)\Delta\Psi dxd\tau+  \epsilon^{2}\int^{t}_{0}\int_{\r}(n-1)\triangle\Psi dxd\tau\notag\\&+4\epsilon^{2}\int^{t}_{0}\int_{\r}(\nabla\sqrt{n}\otimes\nabla\sqrt{n}):\nabla(v+\nabla\Psi)dxd\tau\notag  \\&-4\epsilon^{2}\int_{\r}\nabla\sqrt{n}\cdot\nabla\sqrt{1+\epsilon\sigma}dx+2\epsilon^{2}\int_{\r}|\nabla\sqrt{1+\epsilon\sigma}|^{2}dx\notag  \\&+4\epsilon^{2}\int_{\r}\nabla\sqrt{n_{0,\epsilon}}\cdot\nabla\sqrt{1+\epsilon\sigma_{0}}dx-2\epsilon^{2}\int_{\r}|\nabla\sqrt{1+\epsilon\sigma_{0}}|^{2}dx\notag\\
    &-2\epsilon\int^{t}_{0}\int_{\r}nD(u-v-\nabla\Psi):D(v+\nabla\Psi)dxd\tau\notag\\
    =&\mathcal{E}(0)+{\mathbf{I}}_{1}+{\mathbf{I}}_{2}+{\mathbf{I}}_{3}+{\mathbf{I}}_{4}+{\mathbf{I}}_{5}+{\mathbf{I}}_{6}+{\mathbf{I}}_{7}+{\mathbf{I}}_{8}+{\mathbf{I}}_{9}+{\mathbf{I}}_{10}+{\mathbf{I}}_{11}.\end{align} 
    
    Now, we begin to treat $\mathcal{E}(0)$  and the integrals ${\mathbf{I}}_{k}  (k = 1, 2, \cdots,11)$  term by term.

   For the $\mathcal{E}(0)$, we have
\begin{align}
  \mathcal{E}(0)=  &\frac{1}{2} \int_{\r}n_{0,\epsilon}|u_{0,\epsilon}-u_{0}|^{2} dx+2\epsilon^{2} \int_{\r}|\nabla\sqrt{n_{0,\epsilon}}-\nabla\sqrt{1+\epsilon\sigma_{0}}|^2dx\notag\\
    &+\frac{1}{\epsilon^{2}} \int_{\r}[H(1+n^{1}_{0,\epsilon})-\epsilon H'(1+\epsilon n_{0}^{1})(n_{0,\epsilon}-n_{0}^{1})-H(1+\epsilon n_{0}^{1}))]dx, \label{2018070201}\end{align}
where $u_0=v_{0}+\nabla\Psi_0=\mathbf{H}[u_0]+\nabla\Psi_0$.
   
   From \eqref{2018070301}-\eqref{2018070302}, the first term on the right hand side of \eqref{2018070201} can be estimated as
   \begin{align}
\label{2018070311}
    \frac{1}{2} \int_{\r}n_{0,\epsilon}|u_{0,\epsilon}-u_{0}|^{2} dx&\leq\frac{1}{2} \int_{\r}|1+\epsilon n^{1}_{0,\epsilon}||u_{0,\epsilon}-u_{0}|^{2} dx\notag\\
    &  \leq \frac{1}{2} \int_{\r}|u_{0,\epsilon}-u_{0}|^{2} dx+\frac{\epsilon}{2} \int_{\r}|n^{1}_{0,\epsilon}||u_{0,\epsilon}-u_{0}|^{2} dx\notag\\
    &  \leq \frac{1}{2} \int_{\r}|u_{0,\epsilon}-u_{0}|^{2} dx+\frac{\epsilon}{2} \|n^{1}_{0,\epsilon}\|_{L^{\infty}(\r)}\int_{\r}|u_{0,\epsilon}-u_{0}|^{2} dx\notag\\
    &  \leq C(1+\epsilon)\|u_{0,\epsilon}-u_{0}\|^2_{L^{2}(\r)}.\end{align}
  
  Using \eqref{2018070302} and the 
  Cauchy-Schwarz's inequality, the second term on the hand of  \eqref{2018070201} can be estimated by 
   \begin{align}
\label{2018071601}
&  2\epsilon^{2} \int_{\r}|\nabla\sqrt{n_{0,\epsilon}}-\nabla\sqrt{1+\epsilon\sigma_{0}}|^2dx\notag  \\
   = &  \frac1 2\epsilon^{4} \int_{\r}\left|\frac{\nabla n^{1}_{0,\epsilon}}{\sqrt{1+\epsilon n^{1}_{0,\epsilon}}}-\frac{\nabla n^{1}_{0}}{\sqrt{1+\epsilon n^{1}_{0}}}\right|^2dx\notag\\
\leq &\epsilon^{4} \int_{\r}\left|\frac{\nabla( n^{1}_{0,\epsilon}-n^{1}_{0})}{\sqrt{1+\epsilon n^{1}_{0,\epsilon}}}\right|^2dx+\epsilon^{4} \int_{\r}\left|\nabla n^{1}_{0} \left(\frac{1}{\sqrt{1+\epsilon n^{1}_{0,\epsilon}}}-\frac{1}{\sqrt{1+\epsilon n^{1}_{0,\epsilon}}}\right)\right|^2dx\notag\\
\leq & C\epsilon^{4}\|n_{0,\epsilon}-n_{0}^{1}\|^2_{H^1(\r)}+C\epsilon^{4}. \end{align}

      For the third term on the right hand side of \eqref{2018070201},    
   using the Taylor formula，one gets
   \begin{align}
\label{2018070312}
    &\frac{1}{\epsilon^{2}} \int_{\r}[H(1+n^{1}_{0,\epsilon})-\epsilon H'(1+\epsilon n_{0}^{1})(n_{0,\epsilon}-n_{0}^{1})-H(1+\epsilon n_{0}^{1}))]dx \notag  \\
 \leq   & \frac{C}{\epsilon^{2}} \int_{\r}|\epsilon(n_{0,\epsilon}-n_{0}^{1})|^{2}dx\notag\\
    \leq&C\|n_{0,\epsilon}-n_{0}^{1}\|^2_{L^2(\r)}.\end{align} 
 So, we can get
 \begin{equation}
\label{2018070321}
\mathcal{E}(0)\leq C\left(\|u_{0,\epsilon}-u_{0}\|^2_{L^{2}(\r)}+\|n^{1}_{0,\epsilon}-n^{1}_{0}\|^2_{H^{1}(\r)}+C\epsilon^{4}\right).\end{equation}
 
     For ${\mathbf{I}}_{1}$, we have
    \begin{align}
\label{}
  {\mathbf{I}}_{1}=  &   -\int^{t}_{0}\int_{\r}n (u-v-\nabla\Psi)\cdot\nabla(v+\nabla\Psi)]\cdot(u-v-\nabla\Psi)dxd\tau \notag\\
    &-\int^{t}_{0}\int_{\r}n[\pt(v+\nabla\Psi)+(v+\nabla\Psi)\cdot\nabla(v+\nabla\Psi)]\cdot(u-v-\nabla\Psi)dxd\tau  \notag\\
    =  &   -\int^{t}_{0}\int_{\r}n (u-v-\nabla\Psi)\cdot\nabla(v+\nabla\Psi)]\cdot(u-v-\nabla\Psi)dxd\tau \notag\\
    &-\int^{t}_{0}\int_{\r}nu(\pt v+v\cdot\nabla v)dxd\tau+\int^{t}_{0}\int_{\r}n(v+\nabla\Psi)(\pt v+v\cdot\nabla v)dxd\tau   \notag\\
     &  - \int^{t}_{0}\int_{\r}nu \pt \nabla \Psi dxd\tau +\int^{t}_{0}\int_{\r}nv\pt \nabla \Psi dxd\tau +\frac{1}{2}\int^{t}_{0}\int_{\r}n \pt| \nabla \Psi|^{2} dxd\tau \notag\\
      &-\int^{t}_{0}\int_{\r}n (u-v-\nabla\Psi)\otimes\nabla\Psi: \nabla v dxd\tau\notag\\
      & -\int^{t}_{0}\int_{\r}n (u-v-\nabla\Psi)\otimes  v: \nabla^{2}\Psi   dxd\tau  \notag\\
      &-\frac{1}{2}\int^{t}_{0}\int_{\r}n (u-v-\nabla\Psi)\cdot\nabla|\nabla\Psi|^{2}   dxd\tau \notag\\=&{\mathbf{I}}_{1,1} +{\mathbf{I}}_{1,2} +{\mathbf{I}}_{1,3}+{\mathbf{I}}_{1,4}+{\mathbf{I}}_{1,5}+{\mathbf{I}}_{1,6}+{\mathbf{I}}_{1,7}+{\mathbf{I}}_{1,8}+{\mathbf{I}}_{1,9}.\end{align}
    
  Next, we begin to   estimaet   ${\mathbf{I}}_{1,i}(i=1,2,\cdots,9)$.     
  Using the Sobolev imbedding  theorem, the Minkowski  inequality, \eqref{ie} and the estimate \eqref{2018062512}, one gets 
  \begin{align}
\label{2018070331}
  {\mathbf{I}}_{1,1} \leq & \int^{t}_{0}\mathcal{E}(\tau)\|\nabla(v+\nabla\Psi)\|_{{L^{\infty}(\r)}}d\tau   \notag   \\\leq & \int^{t}_{0}\mathcal{E}(\tau)\|\nabla v\|_{{L^{\infty}(\r)}}d\tau   + \int^{t}_{0}\mathcal{E}(\tau)\|\nabla^{{2}}\Psi\|_{{L^{\infty}(\r)}}d\tau\notag   \\
   \leq &C\int^{t}_{0}\mathcal{E}(\tau)d\tau\end{align}  and
  \begin{align}
\label{2018070332}
  {\mathbf{I}}_{1,2} = &\int^{t}_{0}\int_{\r}\partial_{\tau}(n-1)\Pi dxd\tau\notag   \\
  = & \int_{\r} (n-1)\Pi dx - \epsilon\int_{\r}  n^{1}_{0,\epsilon} \Pi_0 dx-\int^{t}_{0}\int_{\r}(n-1)\partial_{\tau}\Pi dxd\tau\notag   \\
   \leq &\left\|(n-1)\mathbf{1}_{\{|n-1|<1\}}\right\|_{L^2(\r)}\|\Pi\|_{L^2(\r)}\notag\\&+\left\|(n-1)\mathbf{1}_{\{|n-1|\geq1\}}\right\|_{L^\gamma(\r)}\|\Pi\|_{L^{\frac{\gamma}{\gamma-1}}(\r)}+C\epsilon\notag\\&  +C\left\|(n-1)\mathbf{1}_{\{|n-1|<1\}}\right\|_{L^{\infty}([0,T];L^2(\r))}\|\pt\Pi\|_{L^{\infty}([0,T];L^2(\r))}\notag\\&+C\left\|(n-1)\mathbf{1}_{\{|n-1|\geq1\}}\right\|_{L^{\infty}([0,T];L^\gamma(\r))}\|\pt\Pi\|_{L^{\infty}([0,T];L^{\frac{\gamma}{\gamma-1}}(\r))}\notag\\\leq &C\epsilon^{\min\{1,\frac{2}{\gamma}\}}.\end{align}    
 Using  \eqref{2018062512}  and the  interpolation inequality, we have  
    \begin{align}
\label{2018070471}
\|\nabla\Psi\cdot\nabla\Pi\|_{L^{\frac{\gamma}{\gamma-1}}(\r)}\leq & \|\nabla\Psi\cdot\nabla\Pi\|^{\frac{\gamma-1}{\gamma}}_{L^{1}(\r)}\|\nabla\Psi\cdot\nabla\Pi\|^{\frac{1}{\gamma}}_{L^{\infty}(\r)}\notag\\
   \leq &   \|\nabla\Psi\|^{\frac{\gamma-1}{\gamma}}_{L^{2}(\r)} \|\nabla\Pi\|^{\frac{\gamma-1}{\gamma}}_{L^{2}(\r)}  \|\nabla\Psi\|^{\frac{1}{\gamma}}_{L^{\infty}(\r)}\|\nabla\Pi\|^{\frac{1}{\gamma}}_{L^{\infty}(\r)}\notag\\
   \leq&C\left(1+\frac{t}{\epsilon}\right)^{-\frac{1}{\gamma}}.
     \end{align}
So, by using \eqref{2018070122}-\eqref{2018070409}, the acoustic equations \eqref{2018062501}-\eqref{2018062502}  and $\md v=0$, we have
\begin{align}
\label{2018070401}
  {\mathbf{I}}_{1,3}=  -&\int^{t}_{0}\int_{\r}n(v+\nabla\Psi)\cdot \nabla\Pi dxd\tau    \notag  \\
    =& -\int^{t}_{0}\int_{\r}(n-1)(v+\nabla\Psi)\cdot \nabla\Pi dxd\tau-\int^{t}_{0}\int_{\r}(v+\nabla\Psi)\cdot \nabla\Pi dxd\tau\notag\\
    \leq&C\int^{t}_{0}\left\|(n-1)\mathbf{1}_{\{|n-1|<1\}}\right\|_{L^2(\r)}\|v\cdot\nabla\Pi\|_{L^{2}(\r)} d\tau\notag\\
    &+C\int^{t}_{0}\left\|(n-1)\mathbf{1}_{\{|n-1|\geq1\}}\right\|_{L^\gamma(\r)}\|\nabla\Psi\cdot\nabla\Pi\|_{L^{\frac{\gamma}{\gamma-1}}(\r)}d\tau\notag\\&+\int^{t}_{0}\int_{\r}\triangle\Psi \cdot \Pi dxd\tau\notag\\
    \leq&C\epsilon-\epsilon\int^{t}_{0}\int_{\r}\partial_{\tau}s \cdot \Pi dxd\tau\notag\\&+C\int^{t}_{0}\left\|(n-1)\mathbf{1}_{\{|n-1|\geq1\}}\right\|_{L^\gamma(\r)}\|\nabla\Psi\cdot\nabla\Pi\|_{L^{\frac{\gamma}{\gamma-1}}(\r)}d\tau\notag\\
    =&C\epsilon-\epsilon\left[\int_{\r}s \cdot \Pi dx-\int_{\r}n_{0}^1\cdot \Pi _0dx\right]+\epsilon\int^{t}_{0}\int_{\r}s \cdot \partial_{\tau}\Pi dxd\tau\notag\\
    &+C\int^{t}_{0}\left\|(n-1)\mathbf{1}_{\{|n-1|\geq1\}}\right\|_{L^\gamma(\r)}\|\nabla\Psi\cdot\nabla\Pi\|_{L^{\frac{\gamma}{\gamma-1}}(\r)}d\tau\notag\\
   \leq&C\epsilon+C\epsilon^{\frac{2}{\gamma}}\int^{t}_{0}\|\nabla\Psi\cdot\nabla\Pi\|_{L^{\frac{\gamma}{\gamma-1}}(\r)}d\tau\notag\\
   \leq&C\epsilon +C\epsilon^{\frac{2}{\gamma}}\cdot\frac{\gamma\epsilon}{\gamma-1}\left[\left(1+\frac{T}{\epsilon}\right)^{1-\frac{1}{\gamma}}-1\right]\notag\\
\leq& C\epsilon^{\min\{1,\frac{2}{\gamma}\}} ,\end{align}   
where, we used   
\begin{equation}
\label{2018070518}
\lim_{\epsilon\to0}\frac{\epsilon\left(\left(1+\frac{T}{\epsilon}\right)^{1-\frac{1}{\gamma}}-1\right)}{\epsilon^{\frac{1}{\gamma}}}=T^{1-\frac{1}{\gamma}}.
\end{equation} 
 The term ${\mathbf{I}}_{1,4}$   will be canceled by its counterpart in  ${\mathbf{I}}_{3}$. 
   Using   the acoustic equations \eqref{2018062501}-\eqref{2018062502},   with the help of  \eqref{ie},  \eqref{2018062512}, \eqref{2018070122}, \eqref{2018070409} and $\md v=0$, we have    
    \begin{align}
\label{2018070461}
  {\mathbf{I}}_{1,5}=  & \int^{t}_{0}\int_{\r}(n-1)v\partial_{\tau}\nabla \Psi dxd\tau +  \underbrace{\int^{t}_{0}\int_{\r}v\partial_{\tau} \nabla \Psi dxd\tau}_{=0} \notag\\
   = & - \frac{p'(1)}{\epsilon} \int^{t}_{0}\int_{\r}(n-1)v\cdot\nabla \sigma   dxd\tau \notag\\
   \leq &\frac{p'(1)}{\epsilon}\int^{t}_{0}\left\|(n-1)\mathbf{1}_{\{|n-1|<1\}}\right\|_{L^2(\r)}\|v\|_{L^2(\r)} \|\nabla\sigma\|_{L^\infty(\r)}d\tau\notag\\
    &+\frac{p'(1)}{\epsilon}\int^{t}_{0}\left\|(n-1)\mathbf{1}_{\{|n-1|\geq1\}}\right\|_{L^\gamma(\r)}\|v\|_{L^{\frac{\gamma}{\gamma-1}}(\r)} \|\nabla\sigma\|_{L^\infty(\r)}d\tau\notag\\
  \leq  &C\epsilon^{\min\{1,\frac{2}{\gamma}\}}[\ln(\epsilon+T)-\ln\epsilon]\notag\\
  \leq&C\epsilon^{\min\{1-\frac{1}{\gamma},\frac{1}{\gamma}\}},\end{align}
  where we used 
 \begin{align}
\label{2018070517}
 \lim_{\epsilon\to0} \frac{\ln(\epsilon+T)-\ln\epsilon}{\epsilon^{-\frac{1}{\gamma}}} =0.   
\end{align}
From  the acoustic equations \eqref{2018062501}-\eqref{2018062502}, we have\begin{align}
\label{2018070501}
   {\mathbf{I}}_{1,6}= &\frac{1}{2}\int^{t}_{0}\int_{\r}(n-1) \partial_{\tau}| \nabla \Psi|^{2} dxd\tau+\frac{1}{2}\int^{t}_{0}\int_{\r}\partial_{\tau}| \nabla \Psi|^{2} dxd\tau \\
    =&\frac{p'(1) }{\epsilon} \int^{t}_{0}\int_{\r}(n-1) \nabla \Psi\cdot\nabla\sigma dxd\tau+\frac12\left[\int_{\r} | \nabla \Psi|^{2} dx-\int_{\r}| \nabla \Psi_0|^{2} dx\right]\notag\\\leq& \frac{p'(1) }{\epsilon}\int^{t}_{0}\left\|(n-1)\mathbf{1}_{\{|n-1|<1\}}\right\|_{L^2(\r)}\|\nabla \Psi\|_{L^2(\r)} \|\nabla\sigma\|_{L^\infty(\r)}d\tau\notag\\
    &+\frac{p'(1)}{\epsilon}\int^{t}_{0}\left\|(n-1)\mathbf{1}_{\{|n-1|\geq1\}}\right\|_{L^\gamma(\r)}\|\nabla \Psi\|_{L^{\frac{\gamma}{\gamma-1}}(\r)} \|\nabla\sigma\|_{L^\infty(\r)}d\tau\notag\\&+\frac12\left[\int_{\r} | \nabla \Psi|^{2} dx-\int_{\r}| \nabla \Psi_0|^{2} dx\right]\notag\\
  \leq& C \int^{t}_{0}\left (1+\frac{\tau}{\epsilon}\right)^{-1}d\tau +C\epsilon^{\frac{2}{\gamma}-1} \int^{t}_{0}\left (1+\frac{\tau}{\epsilon}\right)^{-1-\frac{1}{\gamma}}d\tau \notag\\&+\frac12\left[\int_{\r} | \nabla \Psi|^{2} dx-\int_{\r}| \nabla \Psi_0|^{2} dx\right]\notag\\\leq &C\epsilon[\ln(\epsilon+T)-\ln\epsilon]+C\epsilon^{\frac{2}{\gamma}}\left[1-\left(\frac{\epsilon}{\epsilon+T}\right)^{\frac{1}{\gamma}}\right]\notag\\&+\frac12\left[\int_{\r} | \nabla \Psi|^{2} dx-\int_{\r}| \nabla \Psi_0|^{2} dx\right]\notag\\
  \leq&C\epsilon^{\min\{1-\frac{1}{\gamma},\frac{2}{\gamma}\}}+\frac12\left[\int_{\r} | \nabla \Psi|^{2} dx-\int_{\r}| \nabla \Psi_0|^{2} dx\right],\end{align}    
 where we have used the following interpolation inequality for $\nabla\Psi$   
    \begin{align}
\label{2018070502}
 \|\nabla \Psi\|_{L^{\frac{\gamma}{\gamma-1}}(\r)}\leq   &
 \|\nabla \Psi\|^{\frac{\gamma-1}{\gamma}}_{L^{1}(\r)}\|\nabla \Psi\|^{\frac{1}{\gamma}}_{L^{\infty}(\r)}\notag\\
 \leq  &C\|\nabla \Psi\|^{\frac{\gamma-1}{\gamma}}_{L^{2}(\r)}\|\nabla \Psi\|^{\frac{1}{\gamma}}_{L^{\infty}(\r)}\notag\\
 \leq&C  \left (1+\frac{t}{\epsilon}\right)^{-\frac{1}{\gamma}}.  \end{align}
For $ {\mathbf{I}}_{1,7} $, noting that
  \begin{align}
\label{2018070511}
    \|nu\|_{L^{\frac{2\gamma}{\gamma+1}}(\r)}\leq&  \|\sqrt{n}\|_{L^{2\gamma}(\r)}   \|\sqrt{n}u\|_{L^{2}(\r)}\leq C,
\end{align} 
     we have 
    \begin{align}
\label{2018070512}
 {\mathbf{I}}_{1,7}   = &-\int^{t}_{0}\int_{\r}n u\otimes\nabla\Psi: \nabla v dxd\tau+\int^{t}_{0}\int_{\r}nv\otimes\nabla\Psi: \nabla v dxd\tau\notag\\&+\int^{t}_{0}\int_{\r}n\nabla\Psi\otimes\nabla\Psi: \nabla v dxd\tau   \notag\\
    \leq&  C\int^{t}_{0}\|nu\|_{L^{\frac{2\gamma}{\gamma+1}}(\r)}\|\nabla\Psi\|_{L^{\frac{2\gamma}{\gamma-1}}(\r)} \|\nabla v\|_{L^{\infty}(\r)}d\tau   \notag\\
    &+C\int^{t}_{0}\|n\|_{L^{1}(\r)}\|v\|_{L^{\infty}(\r)}\|\nabla\Psi\|_{L^{\infty}(\r)} \|\nabla v\|_{L^{\infty}(\r)}d\tau   \notag\\
    &+C\int^{t}_{0}\|n\|_{L^{1}(\r)}\|\nabla\Psi\|^2_{L^{\infty}(\r)} \|\nabla v\|_{L^{\infty}(\r)}d\tau   \notag\\\leq&  C\int^{t}_{0} \left (1+\frac{\tau}{\epsilon}\right)^{-\frac{1}{\gamma}}d\tau  +C\int^{t}_{0} \left (1+\frac{\tau}{\epsilon}\right)^{-1}d\tau +C\int^{t}_{0} \left (1+\frac{\tau}{\epsilon}\right)^{-2}d\tau      \notag\\
    \leq&C \left[\frac{\gamma\epsilon}{\gamma-1}\left(\left(1+\frac{T}{\epsilon}\right)^{1-\frac{1}{\gamma}}-1\right)+C\epsilon(\ln(\epsilon+T)-\ln\epsilon)+\frac{\epsilon T}{\epsilon+T}\right]\notag\\
    \leq &C\epsilon^{\frac{1}{\gamma}},\end{align}    
where, we used   \eqref{2018070517}  and \eqref{2018070518}.
 Similarly to the estimate of  $ {\mathbf{I}}_{1,7} $, we have    
 \begin{align}
\label{2018070513}
    {\mathbf{I}}_{1,8}   = &-\int^{t}_{0}\int_{\r}n u\otimes v: \nabla^{2}\Psi  dxd\tau+\int^{t}_{0}\int_{\r}n\nabla\Psi\otimes v: \nabla^{2}\Psi  dxd\tau\notag\\&+\int^{t}_{0}\int_{\r}n\nabla\Psi\otimes v: \nabla^{2}\Psi dxd\tau   \notag\\
    \leq&C \left[\frac{\gamma\epsilon}{\gamma-1}\left(\left(1+\frac{T}{\epsilon}\right)^{1-\frac{1}{\gamma}}-1\right)+C\epsilon(\ln(\epsilon+T)-\ln\epsilon)+\frac{\epsilon T}{\epsilon+T}\right] \notag\\
    \leq& C\epsilon^{\frac{1}{\gamma}}\end{align}       
  and   
   \begin{align}
\label{2018070516}
 {\mathbf{I}}_{1,9}   = &-\frac12\int^{t}_{0}\int_{\r}n u\cdot\nabla|\nabla\Psi|^{2}dxd\tau+\frac12\int^{t}_{0}\int_{\r}nv\cdot\nabla|\nabla\Psi|^{2} dxd\tau\notag\\&+\frac12\int^{t}_{0}\int_{\r}n\nabla\Psi\cdot\nabla|\nabla\Psi|^{2} dxd\tau    \notag\\
    \leq&C \left[\frac{\gamma\epsilon}{\gamma-1}\left(\left(1+\frac{T}{\epsilon}\right)^{1-\frac{1}{\gamma}}-1\right)+\frac{\epsilon T}{\epsilon+T}+\frac{\epsilon}{2} \right] \notag\\
    \leq& C\epsilon^{\frac{1}{\gamma}}.\end{align}
    Therefore, we obtain that
    \begin{align}
\label{2018070521}
  {\mathbf{I}}_{1}\leq&  C\int^{t}_{0}\mathcal{E}(\tau)d\tau+C\epsilon^{\min\{1-\frac{1}{\gamma},\frac{1}{\gamma}\}}\notag\\&+\frac12\left[\int_{\r} | \nabla \Psi|^{2} dx-\int_{\r}| \nabla \Psi_0|^{2} dx\right]- \int^{t}_{0}\int_{\r}nu \pt \nabla \Psi dxd\tau. \end{align}

  For ${\mathbf{I}}_{2}$, using $H''(1)=p'(1)$ and observing that
  \begin{equation}\label{2018070701}
 H''(1+\epsilon\sigma)-H''(1)=\epsilon H'''(\xi)\sigma,\quad \xi \in(1,1+\epsilon\sigma)\; \mbox{or}\; (1+\epsilon\sigma,1),\end{equation}
 one gets        
  \begin{align}
\label{2018070522}
  {\mathbf{I}}_{2}  =&- \frac{1}{\epsilon}\int^{t}_{0}\int_{\r}(n-1-\epsilon\sigma) H''(1+\epsilon\sigma)\partial_{\tau}\sigma dxd\tau \notag  \\
    =&  -\frac{1}{\epsilon^{2}}\int^{t}_{0}\int_{\r}(n-1) [H''(1+\epsilon\sigma)-H''(1)]\triangle\Psi dxd\tau \notag  \\
&+\frac{1}{\epsilon^{2}}\int^{t}_{0}\int_{\r}(n-1) H''(1)\triangle\Psi dxd\tau\notag\\&+\int^{t}_{0}\int_{\r}\sigma[H''(1+\epsilon\sigma)-H''(1)]\partial_{\tau} \sigma dxd\tau+\int^{t}_{0}\int_{\r}\sigma H''(1) \partial_{\tau} \sigma dxd\tau\notag\\
=& -\frac{1}{\epsilon}\int^{t}_{0}\int_{\r}(n-1) H'''(\xi)\sigma\triangle\Psi dxd\tau+\frac{p'(1)}{\epsilon^{2}}\int^{t}_{0}\int_{\r}(n-1) \triangle\Psi dxd\tau \notag  \\
&-\int^{t}_{0}\int_{\r}  H'''(\xi) \sigma^{2} \triangle\Psi dxd\tau+p'(1)\int^{t}_{0}\int_{\r}\sigma \partial_{\tau} \sigma dxd\tau\notag\\
\leq&\frac{1}{\epsilon}\int^{t}_{0}\left\|(n-1)\mathbf{1}_{\{|n-1|<1\}}\right\|_{L^2(\r)} \|H'''(\xi) \|_{L^2(\r)}\|\sigma\|_{L^\infty(\r)}  \|\triangle\Psi\|_{L^\infty(\r)}d\tau\notag\\
&+ \frac{1}{\epsilon}\int^{t}_{0}\left\|(n-1)\mathbf{1}_{\{|n-1|\geq1\}}\right\|_{L^\gamma(\r)} \| H'''(\xi) \|_{L^\frac{\gamma}{\gamma-1}(\r)} \|\sigma\|_{L^\infty(\r)} \|\triangle\Psi\|_{L^\infty(\r)}d\tau\notag\\
&+C\int^{t}_{0} \|\sigma\|^2_{L^\infty(\r)} \|\triangle\Psi\|_{L^\infty(\r)} d\tau+\frac{p'(1)}{\epsilon^{2}}\int^{t}_{0}\int_{\r}(n-1) \triangle\Psi dxd\tau \notag\\
&+ \frac{1}{2}\left[\int_{\r}p'(1) |\sigma|^{2} dx-\int_{\r}p'(1)| \sigma_0|^{2} dx\right]\notag\\
\leq&C\int^{t}_{0}\left (1+\frac{\tau}{\epsilon}\right)^{-2}d\tau  +C\epsilon^{\frac{2}{\gamma}-1}\int^{t}_{0}\left (1+\frac{\tau}{\epsilon}\right)^{-2}d\tau+C\int^{t}_{0}\left (1+\frac{\tau}{\epsilon}\right)^{-3}d\tau  \notag\\&+\frac{p'(1)}{\epsilon^{2}}\int^{t}_{0}\int_{\r}(n-1) \triangle\Psi dxd\tau + \frac{1}{2}\left[\int_{\r}p'(1) |\sigma|^{2} dx-\int_{\r}p'(1)| \sigma_0|^{2} dx\right]\notag\\
\leq&C\left(\frac{\epsilon T}{\epsilon+T}+\frac{\epsilon^{\frac{2}{\gamma}} T}{\epsilon+T}+\frac{2T\epsilon^{2}+T^{2}\epsilon}{2(\epsilon+T)^{2}}\right)+\frac{p'(1)}{\epsilon^{2}}\int^{t}_{0}\int_{\r}(n-1) \triangle\Psi dxd\tau \notag\\&+\frac{1}{2}\left[\int_{\r}p'(1) |\sigma|^{2} dx-\int_{\r}p'(1)| \sigma_0|^{2} dx\right]\notag\\\leq&C\epsilon^{\min\{1,\frac{2}{\gamma}\}} +\frac{p'(1)}{\epsilon^{2}}\int^{t}_{0}\int_{\r}(n-1) \triangle\Psi dxd\tau \notag\\&+\frac{1}{2}\left[\int_{\r}p'(1) |\sigma|^{2} dx-\int_{\r}p'(1)| \sigma_0|^{2} dx\right].\end{align}
    
 Using  \eqref{2018070518},\eqref{2018070701} and the acoustic equations \eqref{2018062501}-\eqref{2018062502}, we have 
  \begin{align}
\label{2018070702}
    {\mathbf{I}}_{3}= & - \frac{1}{\epsilon}\int^{t}_{0}\int_{\r}nu\cdot   H''(1+\epsilon\sigma)\nabla\sigma dxd\tau \notag \\
   = & - \frac{1}{\epsilon}\int^{t}_{0}\int_{\r}nu\cdot  (H''(1+\epsilon\sigma)-H''(1))\nabla\sigma dxd\tau \notag\\& -\frac{1}{\epsilon}\int^{t}_{0}\int_{\r}nu\cdot  H''(1)\nabla\sigma dxd\tau \notag\\
   =&- \int^{t}_{0}\int_{\r}nu\cdot  H'''(\xi)\nabla\sigma dxd\tau  +\int^{t}_{0}\int_{\r}nu\partial_{\tau}   \nabla\Psi dxd\tau \notag\\
   \leq &C\int^{t}_{0}\|nu\|_{L^{\frac{2\gamma}{\gamma+1}}(\r)}\|\nabla\sigma\|_{L^{\frac{2\gamma}{\gamma-1}}(\r)}d\tau+\int^{t}_{0}\int_{\r}nu\partial_{\tau}   \nabla\Psi dxd\tau\notag\\\leq&C\int^{t}_{0} \left (1+\frac{\tau}{\epsilon}\right)^{-\frac{1}{\gamma}}d\tau  +\int^{t}_{0}\int_{\r}nu\partial_{\tau}   \nabla\Psi dxd\tau\notag\\\leq&C \left[\frac{\gamma\epsilon}{\gamma-1}\left(\left(1+\frac{T}{\epsilon}\right)^{1-\frac{1}{\gamma}}-1\right)\right] +\int^{t}_{0}\int_{\r}nu\partial_{\tau}   \nabla\Psi dxd\tau\notag\\
    \leq& C\epsilon^{\frac{1}{\gamma}}+\int^{t}_{0}\int_{\r}nu\partial_{\tau}   \nabla\Psi dxd\tau.\end{align}    
    
    In view of \eqref{2017070711}, we deduce that
 \begin{align}
\label{2018070703}
  {\mathbf{I}}_{4}=  &  -\frac{1}{\epsilon^2} \int^{t}_{0}\int_{\r}\left[p(n)-p'(1)(n-1)-p(1)\right]\Delta\Psi dxd\tau\notag\\
  &-\frac{p'(1)}{\epsilon^{2}}\int^{t}_{0}\int_{\r}(n-1) \triangle\Psi dxd\tau -\underbrace{\frac{1}{\epsilon^2} \int^{t}_{0}\int_{\r}p(1)\Delta\Psi dxd\tau}_{=0}\notag\\
   = & -\frac{\gamma-1}{\epsilon^2} \int^{t}_{0}\int_{\r}H(n)\Delta\Psi dxd\tau-\frac{p'(1)}{\epsilon^{2}}\int^{t}_{0}\int_{\r}(n-1) \triangle\Psi dxd\tau\notag\\
 = &-\frac{\gamma-1}{\epsilon^2} \int^{t}_{0}\int_{\r}[H(n)-H'(1+\epsilon\sigma)(n-1-\epsilon \sigma)-H(1+\epsilon \sigma)]\Delta\Psi dxd\tau\notag\\
  &-\frac{\gamma-1}{\epsilon^2} \int^{t}_{0}\int_{\r}[H'(1+\epsilon\sigma)(n-1)]\Delta\Psi dxd\tau\notag\\&+\frac{\gamma-1}{\epsilon}\int^{t}_{0}\int_{\r}H'(1+\epsilon\sigma)\sigma\Delta\Psi dxd\tau-\frac{\gamma-1}{\epsilon^2} \int^{t}_{0}\int_{\r}H(1+\epsilon \sigma)\Delta\Psi dxd\tau\notag\\&-\frac{p'(1)}{\epsilon^{2}}\int^{t}_{0}\int_{\r}(n-1) \triangle\Psi dxd\tau\notag\\
  =&{\mathbf{I}}_{4,1}+{\mathbf{I}}_{4,2}+{\mathbf{I}}_{4,3}+{\mathbf{I}}_{4,4}-\frac{p'(1)}{\epsilon^{2}}\int^{t}_{0}\int_{\r}(n-1) \triangle\Psi dxd\tau.\end{align}   
   For ${\mathbf{I}}_{4,1}$, we have 
   \begin{equation}
\label{2018070712}
{\mathbf{I}}_{4,1}\leq C\int^{t}_{0}\mathcal{E}(\tau)d\tau.\end{equation} 
   Noting that 
   \begin{equation}
\label{2018070714}
(1+\epsilon\sigma)^{\gamma-1}-1=(\gamma-1)\epsilon\zeta^{\gamma-2}\sigma,\quad \zeta\in(1,1+\epsilon\sigma)\; \mbox{or}\; (1+\epsilon\sigma,1),\end{equation}
  we have   
  \begin{align}
\label{2018070713}
   {\mathbf{I}}_{4,2}= & -\frac{\gamma}{\epsilon^2} \int^{t}_{0}\int_{\r}[(1+\epsilon\sigma)^{\gamma-1}-1](n-1)\Delta\Psi dxd\tau \notag \\
   = &   -\frac{\gamma(\gamma-1)}{\epsilon} \int^{t}_{0}\int_{\r}\zeta^{\gamma-2}(n-1)\sigma\Delta\Psi dxd\tau \notag \\
   \leq&\frac{C}{\epsilon}\int^{t}_{0}\left\|(n-1)\mathbf{1}_{\{|n-1|<1\}}\right\|_{L^2(\r)} \|\sigma\|_{L^2(\r)}  \|\triangle\Psi\|_{L^\infty(\r)}d\tau\notag\\
&+ \frac{C}{\epsilon}\int^{t}_{0}\left\|(n-1)\mathbf{1}_{\{|n-1|\geq1\}}\right\|_{L^\gamma(\r)} \|\sigma\|_{L^\frac{\gamma}{\gamma-1}(\r)} \|\triangle\Psi\|_{L^\infty(\r)}d\tau\notag\\\leq&C\epsilon^{\min\{1-\frac{1}{\gamma},\frac{1}{\gamma}\}}.\end{align}
  Using \eqref{2018062512},\eqref{2018070517} and\eqref{2018070714}, we  get
  \begin{align}
\label{2018070721}
 {\mathbf{I}}_{4,3}=   &  -\frac{\gamma}{\epsilon} \int^{t}_{0}\int_{\r}[(1+\epsilon\sigma)^{\gamma-1}-1]\sigma\Delta\Psi dxd\tau \notag \\
   = &   -\gamma(\gamma-1) \int^{t}_{0}\int_{\r}\zeta^{\gamma-2}\sigma^{2}\Delta\Psi dxd\tau \notag \\
   \leq&C\int^{t}_{0}\|\sigma\|^{2}_{L^2(\r)}  \|\triangle\Psi\|_{L^\infty(\r)}d\tau\notag\\
\leq&C\epsilon^{1-\frac{1}{\gamma}}.\end{align}  
  Observing that 
  \begin{equation}
\label{2018070725}
(1+\epsilon\sigma)^{\gamma}=1+\gamma\epsilon\sigma+\frac{\gamma(\gamma-1)}{2}\omega^{\gamma-2}\epsilon^{2}\sigma^{2},\quad \omega\in(1,1+\epsilon\sigma)\; \mbox{or}\; (1+\epsilon\sigma,1),\end{equation} we have
 \begin{align}
\label{2018070726}
  {\mathbf{I}}_{4,4}=&  -\frac{1}{\epsilon^2} \int^{t}_{0}\int_{\r}[(1+\epsilon\sigma)^{\gamma}-1-\gamma\epsilon\sigma]\Delta\Psi dxd\tau\notag\\=&-\frac{\gamma(\gamma-1)}{2}\int^{t}_{0}\int_{\r}\omega^{\gamma-2}\sigma^{2}\Delta\Psi dxd\tau\notag\\\leq&C\int^{t}_{0}\|\sigma\|^{2}_{L^2(\r)}  \|\triangle\Psi\|_{L^\infty(\r)}d\tau\notag\\
\leq&C\epsilon^{1-\frac{1}{\gamma}}.\end{align}    
   Then, we have
    \begin{align}
\label{2018070728}
    {\mathbf{I}}_{4}\leq C\int^{t}_{0}\mathcal{E}(\tau)d\tau+C\epsilon^{\min\{1-\frac{1}{\gamma},\frac{1}{\gamma}\}}-\frac{p'(1)}{\epsilon^{2}}\int^{t}_{0}\int_{\r}(n-1) \triangle\Psi dxd\tau.\end{align}
  For ${\mathbf{I}}_{5}$, we get 
   \begin{align}
\label{2018070726}
  {\mathbf{I}}_{5}\leq& C\epsilon^{2}\int^{t}_{0}\left\|(n-1)\mathbf{1}_{\{|n-1|<1\}}\right\|_{L^2(\r)}\|\triangle\Psi\|_{L^2(\r)}d\tau\notag\\
&+ C\epsilon^{2}\int^{t}_{0}\left\|(n-1)\mathbf{1}_{\{|n-1|\geq1\}}\right\|_{L^\gamma(\r)}  \|\triangle\Psi\|_{L^{\frac{\gamma}{\gamma-1}}(\r)}d\tau\notag\\
\leq&C\epsilon^{3}+C\epsilon^{2+\frac{2}{\gamma}}\int^{t}_{0}  \left (1+\frac{t}{\epsilon}\right)^{-\frac{1}{\gamma}}d\tau\notag\\
\leq&C\epsilon^{3}+C\epsilon^{2+\frac{3}{\gamma}},\end{align}    
    where we used \eqref{2018070518} and the following interpolation inequality for $\triangle\Psi$   
    \begin{align}
\label{2018070502}
 \|\triangle \Psi\|_{L^{\frac{\gamma}{\gamma-1}}(\r)}\leq   &
 \|\triangle\Psi\|^{\frac{\gamma-1}{\gamma}}_{L^{1}(\r)}\|\triangle \Psi\|^{\frac{1}{\gamma}}_{L^{\infty}(\r)}\notag\\
 \leq  &C\|\triangle\Psi\|^{\frac{\gamma-1}{\gamma}}_{L^{2}(\r)}\|\triangle \Psi\|^{\frac{1}{\gamma}}_{L^{\infty}(\r)}\notag\\
 \leq&C  \left (1+\frac{t}{\epsilon}\right)^{-\frac{1}{\gamma}}.  \end{align}
    For ${\mathbf{I}}_{6}$, we have    
    \begin{align}
\label{2018070801}
    {\mathbf{I}}_{6} \leq  C\int^{t}_{0}\mathcal{E}(\tau)d\tau.\end{align}
   
  Using \eqref{qmhdener},\eqref{2018062512} and the Young  inequality, we have 
  \begin{align}
\label{2018071610}
 {\mathbf{I}}_{7} = &-2\epsilon^{3}\int_{\r}\frac{\nabla \sqrt{n}\cdot\nabla\sigma}{\sqrt{1+\epsilon\sigma}}dx\leq C\epsilon^{3}\left(\int_{\r}|\nabla \sqrt{n}|^{2}dx+\int_{\r}|\nabla \sigma|^{2}dx\right)\leq C\epsilon.\end{align}
   
  For ${\mathbf{I}}_{8}$, we get 
  \begin{align}
\label{2018071611}
   {\mathbf{I}}_{8} = 2\epsilon^{4}\int_{\r} \left|\frac{\nabla\sigma}{\sqrt{1+\epsilon\sigma}}\right |^{2}dx \leq C\epsilon^{4}.\end{align} 
 
  Using \eqref{2018062512} and the Young  inequality, we have  
    \begin{align}
\label{2018071612}
 {\mathbf{I}}_{9} =& \epsilon^{4}\int_{\r}\frac{\nabla n_{0,\epsilon}^{1}\cdot\nabla n_{0}^{1}}{ \sqrt{(1+\epsilon n_{0,\epsilon}^{1})(1+\epsilon n_{0}^{1})}}dx\notag\\ \leq&C\epsilon^{4}\left(\int_{\r}|\nabla n_{0,\epsilon}^{1}|^{2}dx+\int_{\r}|\nabla n_{0}^{1}|^{2}dx\right)\notag\\\leq &C\epsilon^{4}.
\end{align}    
   
   For $ {\mathbf{I}}_{10}$, we have
   \begin{align}
\label{2018071613}
 {\mathbf{I}}_{10} =-2\epsilon^{4}\int_{\r}\left|\frac{\nabla n_{0}^{1}}{\sqrt{1+\epsilon n_{0}^{1}}}\right|^{2}dx\leq C\epsilon^{4}.\end{align}
   
     For the last term ${\mathbf{I}}_{11}$,   using the Young  inequality with $\epsilon$, one gets
   \begin{align}
\label{2018070802}
   {\mathbf{I}}_{11} =&-2\epsilon\int^{t}_{0}\int_{\r}nD(u-v-\nabla\Psi):D(v+\nabla\Psi)dxd\tau\notag \\  \leq & C \epsilon\int^{t}_{0}\int_{\r}|\sqrt{n}|^{2}dxd\tau+\epsilon\int^{t}_{0}\int_{\r}|\sqrt{n}D(u-v-\nabla\Psi)|^{2}dxd\tau\notag\\
   \leq & C \epsilon+\epsilon\int^{t}_{0}\int_{\r}n|D(u-v-\nabla\Psi)|^{2}dxd\tau.\end{align} 
     Using \eqref{2018062511} and adding the estimates for $\mathcal{E}(0),{\mathbf{I}}_{1},{\mathbf{I}}_{2},\cdots,{\mathbf{I}}_{11}$, some integrals cancel, and we end up with  
     \begin{align}
\label{2018070866}
\mathcal{E}(t)    \leq&C\left(\|u_{0,\epsilon}-u_{0}\|^2_{L^{2}(\r)}+\|n^{1}_{0,\epsilon}-n^{1}_{0}\|^2_{H^{1}(\r)}+ C\epsilon^{\min\{1-\frac{1}{\gamma},\frac{1}{\gamma}\}}\right)+C\int^{t}_{0}\mathcal{E}(\tau)d\tau,\end{align}  
   where we have used the following fact
   \begin{equation*}
\epsilon^{a}>\epsilon^{b},\quad 0<a<b
\end{equation*}
for any $\epsilon\in(0,1)$.   The integral form of the Gronwall inequality gives
   \begin{align}
\label{2018070901}
  \mathcal{E}(t) \leq  &C\left(\|u_{0,\epsilon}-u_{0}\|^2_{L^{2}(\r)}+\|n^{1}_{0,\epsilon}-n^{1}_{0}\|^2_{H^{1}(\r)}+ \epsilon^{\min\{1-\frac{1}{\gamma},\frac{1}{\gamma}\}}\right), \; t\in[0,T],\end{align}  
    which implies that we have proved Theorem \ref{th2018070308}.  Finally, we conclude by observing that Corollary  \ref{2018070906}   follows
from the inequality  \eqref {2018071401}, combined with the estimates  \eqref{2018070512}. In fact, using the Young inequality and Lemma \ref{Lem:2018070121},we have
\begin{align*}
  \|P(\sqrt{n}u)-v\|^{2}_{L^{2}(\r)}  =& \|P(\sqrt{n}u-v-\nabla\Psi)\|^{2}_{L^{2}(\r)}  \\
    \leq&   \|\sqrt{n}u-v-\nabla\Psi\|^{2}_{L^{2}(\r)}  \\
    =& \|\sqrt{n}(u-v-\nabla\Psi)+(\sqrt{n}-1)(v+\nabla\Psi)\|^{2}_{L^{2}(\r)}  \\
    \leq& 2\|\sqrt{n}(u-v-\nabla\Psi)\|^{2}_{L^{2}(\r)}+2\|(\sqrt{n}-1)(v+\nabla\Psi)\|^{2}_{L^{2}(\r)}\\
    \leq&C\mathcal{E}(t) +2\|(\sqrt{n}-1)(v+\nabla\Psi)\|^{2}_{L^{2}(\r)}\\
    \leq &C\mathcal{E}(t) +C\left\|(n-1)\mathbf{1}_{\{|n-1|<1\}}\right\|^2_{L^2(\r)}\\& +C\left\|(n-1)\mathbf{1}_{\{|n-1|\geq1\}}\right\|^{\gamma}_{L^\gamma(\r)}\\\leq&C\epsilon^{\min\{1-\frac{1}{\gamma}\}},\end{align*}
     we have used the following elementary inequality   
    \begin{equation*}
|\sqrt{x}-1|^2\leq C|x-1|^{k},\quad k\geq1
\end{equation*}
for some positive constant $C$ and any $x\geq0$.   

       \medskip
 \noindent
{\sc Acknowledgements:}   J. Yang's research
was partially supported by the Joint Funds of the National Natural
Science Foundation of China (Grant No. U1204103).


\end{document}